\documentclass{svproc}

\usepackage{url}
\usepackage{amsmath}
\usepackage{amssymb}

\usepackage{graphicx}
\usepackage{xcolor}
\usepackage{booktabs}

\begin{document}
\mainmatter              

\title{High-order reliable numerical methods for epidemic models with non-constant recruitment rate}
\titlerunning{Numerical methods for epidemic models with non-constant recruitment rate}   
\author{Bálint Máté Takács\inst{1,2}, Gabriella Svantnerné Sebestyén\inst{1}  \and 
István Faragó\inst{1,2}}
\authorrunning{Bálint Máté Takács et al.} 
\institute{Department of Analysis and Operations Research
Institute of Mathematics
Budapest University of technology and Economics
Műegyetem rkp. 3., H-1111 Budapest, Hungary\\
\email{takacsbm@math.bme.hu}
\and
HUN-REN-ELTE Numerical Analysis and Large Networks Research Group, Hungary}

\maketitle              

\begin{abstract}
The mathematical modeling of the propagation of illnesses has an important role from both mathematical and biological points of view. In this article, we observe an SEIR-type model with a general incidence rate and a non-constant recruitment rate function. First, we observe the qualitative properties of different methods: first-order and higher-order strong stability preserving Runge-Kutta methods \cite{shu}. We give different conditions under which the numerical schemes behave as expected. Then, the theoretical results are demonstrated by some numerical experiments.
\end{abstract}
\section{Introduction}
In recent years it turned out that the mathematical modeling of the propagation of illnesses is not only useful and important from a scientific point of view, but the outbreak of COVID-19 made the role of its real-life applications apparent: it showed the importance of adequate modeling of the illnesses not only from a clinical but also from an economical standpoint. 

One of the key tools to observe such phenomena is mathematical models: in this paper, we consider differential equations. These models can be used to observe the rate of the spread of the illness and can also give an approximation of the number of infected people. Kermack and McKendrick \cite{sir1927} used compartmental models to describe these phenomena. In this case, the members of the population are divided into compartments e.g. suspected or infected, and the members of the population can move between these compartments. In this paper we use the following common compartments:
\begin{itemize}
\item $S(t)$ is the number of susceptible people who are not yet infected at time $t$,
\item $E(t)$ is the number of exposed individuals, who have been infected but are not yet infectious at time $t$,
\item $I(t)$ is the number of infectious people who can spread infectious diseases at time $t$,
\item $R(t)$ is the number of removed or deceased individuals, who have been infected and then removed from the disease meaning that they can no longer infect others at time $t$.
\end{itemize}
In the last decades numerous authors expanded the models of Kermack and McKendrick \cite{matbio}, \cite{capasso}, \cite{sir_bev}. In recent years, the increasing rate of migration in Europe also influenced the rise of models incorporating the role of the movement of a considerable amount of people into the observed region \cite{migration2}, \cite{migration1}. In this paper, we also investigate an epidemic model which can be used to describe such phenomena. Throughout the article, we will assume that the recruitment rate function describing such migration (and also births in the system) is non-constant. 

In Section \ref{sec:cont_mod} we propose the continuous model and show that the solution of this differential equation attains solutions that behave in a biologically reasonable way (e.g. they are non-negative). Then, in Section \ref{sec:discrete} we observe appropriate numerical schemes and show that by using a sufficiently small time step, the schemes preserve the qualitative properties of the continuous model such as positivity, boundedness, and conservation laws. Later in Section \ref{sec:RK} similar theorems are also proved for higher-order methods. Finally, in Section \ref{sec:num} the previous results are demonstrated by some numerical experiments.

\section{Model formulation and preliminaries}\label{sec:cont_mod}
Let us consider the following generalization of the well-known SEIR type epidemics model:
\begin{equation} \label{seir}
\begin{aligned}
\frac{d S (t)}{d t} &= \Pi(t) - \mu S(t) - f(I(t))\cdot S(t), \\
\frac{d E(t)}{d t} &= f(I(t))\cdot S(t) - (\mu + \sigma)E(t), \\
\frac{d I(t)}{d t} &= \sigma E(t) - (\mu + \gamma)I(t) + \delta R(t), \\
\frac{d R(t)}{d t} &= \gamma I(t) - (\mu + \delta) R(t)
\end{aligned}
\end{equation}
$(t \geq 0)$ with the following initial conditions
\begin{align} \label{seir_initial}
S(0) = S_0 \geq 0, \quad E(0) = E_0 \geq 0, \quad I(0) = I_0 \geq 0, \quad R(0) = R_0 \geq 0. 
\end{align}
Moreover, let function $N(t)$ denote the total number of the population, meaning that
\begin{align}\label{eq:N_def}
N(t) = S(t) + E(t) + I(t) + R(t).
\end{align}
The assumptions of the model are the following:
\begin{itemize}
\item There is vital dynamics in this model: the function $\Pi(t)$ is the recruitment rate (incorporating the births and the migration of the model) and the natural deaths rate equals $\mu\geq 0$ in every compartment. Later we will assume that $0 \leq \Pi(t) \leq K$ ($K \in \mathbb{R}^+$) is a continuous function.
\item The incidence function has the form $f(I)\cdot S$, which function describes the way of transmission of the disease between susceptible and infected individuals in the population.
\item The parameter $\gamma$ denotes the recovery rate of the infected individuals.
\item The rate at which exposed individuals become infected is expressed by parameter $\sigma$.
\item The parameter $\delta$ represents the rate at which an individual from the recovered compartments comes to the infected class.
\end{itemize}
The function $f(I)$ describes the effect infected people have on healthy members of the population. In epidemiology, this function is called the force of infection, which expresses the rate at which susceptible individuals acquire an infectious disease \cite{matbio}. From now on, we assume that the function $f(x): \mathbb{R} \rightarrow \mathbb{R}$ has the following properties:
\begin{itemize}
    \item[(C1)] $f(x) \geq 0$ for $x \geq 0$,
    \item[(C2)] $f(0) = 0$,
    \item[(C3)] $f(x)$ is bounded in the sense that there exists a constant $\alpha \in \mathbb{R}^+$ such that $|f(x)|\leq \alpha |x|$ for every $x \in \mathbb{R}$,
    \item[(C4)] $f(x)$ is locally Lipschitz, meaning that every point $z \in \mathrm{Dom}(f)$ has a neighborhood $U_z \in \mathrm{Dom}(f)$ in which $f$ is Lipschitz continuous, i.e. for every $x,y \in U_z$, $|f(x)-f(y)|< L |x-y|$ holds for some constant $L>0$.
\end{itemize} 
Note that these assumptions hold for several different choices of function $f$. The simplest choice of the function $f$ when 
\begin{align}
f(x) =  x.  
\end{align}
Then the interaction between the suspected and infected people can be described by the function $I\cdot S$, which was used in the model \cite{sir1927}. The function of the force of infection can be described with nonlinear functions too. We assume that $f$ is a Holling--type function and has the form
\begin{align}
f(x) = \frac{c_1x}{1 + c_2x^k}
\end{align}
where the parameters show the effect of the infectious \cite{holling1}, \cite{holling2}. We can investigate the effect of the media with the following function \cite{mediaeffect}
\begin{align}
f(x) = \nu e^{-\eta x} x.
\end{align}

The attributes of problem \eqref{seir}--\eqref{seir_initial} are listed in the following theorem.
\begin{theorem} Let us assume that properties (C1)--(C4) hold for function $f$.
\begin{itemize}
\item[a)] The solution of the initial value problem \eqref{seir}--\eqref{seir_initial} uniquely exists for all $t \geq 0$.
\item[b)] For function $N(t)$ defined in \eqref{eq:N_def}, if $\mu>0$ we have \linebreak
$N(t) \leq N(0)+\dfrac{K}{\mu}$ and also
\begin{equation}
    \lim_{t \rightarrow \infty} \left|N(t) -       \dfrac{\Pi(t)}{\mu}\right|=0.
\end{equation}
If $\mu=0$, then $N(t)\leq N(0)+K t$ and if we also assume that $\displaystyle\int_0^{\infty} \Pi(s) ds = \mathcal{P}$, then $\displaystyle\lim_{t \rightarrow \infty} N(t) = N(0) + \mathcal{P}$. 
\item[c)] The solution of the problem \eqref{seir} is non--negative, that is if \linebreak $S(0), \ E(0), \ I(0), \ R(0) \geq0$, then $S(t), \ E(t), \ I(t), \ R(t) \geq 0$, for all $t > 0$.
\end{itemize}   
\end{theorem}

\begin{proof}
\begin{itemize}  
\item[a)] For the existence of unique solutions, it is enough to prove that the right-hand side of the equation has the local Lipschitz property in its second variable and it is continuous. 
Let us write equation \eqref{seir} in the more simple form
\begin{align}
   u'(t)=F(t,u(t)),
\end{align}
in which $u(t)=(S(t),E(t),I(t),R(t))$ and the function $F:\mathbb{R}\times \mathbb{R}^4 \rightarrow \mathbb{R}^4$ is defined as 
\begin{align}
   F(t,S(t),E(t),I(t),R(t))=\left( \begin{array}{c}
       \Pi(t) - \mu S(t) - f(I(t))\cdot S(t) \\
      f(I(t))\cdot S(t) - (\mu + \sigma)E(t) \\
       \sigma E(t) - (\mu + \gamma)I(t) + \delta R(t) \\
       \gamma I(t) - (\mu + \delta) R(t)
   \end{array} \right).
\end{align}
Since $\Pi(t)$ is continuous, we only have to verify the local Lipschitz continuity (where we take the 1-norm of $F$):
   \begin{equation}\label{eq:F_ineq}
      \begin{aligned}
       \Vert F(t,u_1)-F(t,u_2)\Vert = \Vert - \mu S_1 - f(I_1)\cdot S_1  + \mu S_2 + f(I_2)\cdot S_2 \Vert + \\
       +\Vert f(I_1)\cdot S_1 - (\mu + \sigma)E_1 - f(I_2)\cdot S_2 + (\mu + \sigma)E_2 \Vert + \\
       +\Vert \sigma E_1 - (\mu + \gamma)I_1 + \delta R_1 - \sigma E_2 + (\mu + \gamma)I_2 - \delta R_2 \Vert +\\
       +\Vert \gamma I_1 - (\mu + \delta) R_1 - \gamma I_2 + (\mu + \delta) R_2 \Vert
   \end{aligned} 
   \end{equation}    
   For the first term of \eqref{eq:F_ineq}, we get
   \begin{equation}
   \begin{aligned}
       \Vert - \mu S_1 - f(I_1)\cdot S_1 + \mu S_2 + f(I_2)\cdot S_2 \Vert \leq \\
       \leq \mu \Vert S_1 - S_2 \Vert + \Vert f(I_1)\cdot S_1 - f(I_1)\cdot S_2 + f(I_1\cdot S_2 - f(I_2)\cdot S_2\Vert \leq\\
       \leq \mu \Vert S_1 - S_2 \Vert +\Vert f(I_1) \Vert \Vert S_1 - S_2 \Vert + \Vert f(I_1) - f(I_2)\Vert \Vert S_2\Vert \leq
       \end{aligned} 
   \end{equation}
   Now we use conditions (C3) and (C4):
\begin{equation}
       \leq \mu \Vert S_1 - S_2 \Vert +\alpha\Vert I_1 \Vert \Vert S_1 - S_2 \Vert + L\Vert I_1 - I_2\Vert \Vert S_2\Vert \leq
   \end{equation}
   Since we consider the local Lipschitz property, we can assume that $\Vert I_1 \Vert \leq M $ and $\Vert S_2 \Vert \leq M $ for some constant $M>0$, meaning that
   \begin{equation}
       \leq \mu \Vert S_1 - S_2 \Vert +\alpha M \Vert S_1 - S_2 \Vert + L M \Vert I_1 - I_2\Vert. 
    \end{equation}
   The second term of \eqref{eq:F_ineq} can be transformed similarly as the first one, and in the end we get
    \begin{equation}
   \begin{aligned}
       \Vert f(I_1)\cdot S_1 - (\mu + \sigma)E_1 - f(I_2)\cdot S_2 + (\mu + \sigma)E_2 \Vert \leq \\
       \leq (\mu + \sigma)\Vert E_1 - E_2 \Vert + \alpha M \Vert S_1 - S_2 \Vert + L M \Vert I_1 - I_2\Vert. 
   \end{aligned}
   \end{equation}
   Similarly, for the third term we have
   \begin{equation}
   \begin{aligned}
       \Vert \sigma E_1 - (\mu + \gamma)I_1 + \delta R_1 - \sigma E_2 + (\mu + \gamma)I_2 - \delta R_2 \Vert \leq \\
       \leq \sigma \Vert E_1 - E_2\Vert + (\mu + \gamma) \Vert I_1 - I_2\Vert + \delta \Vert R_1 - R_2 \Vert .
    \end{aligned}
    \end{equation}
   The last, fourth term of \eqref{eq:F_ineq} can also be bounded from above:
   \begin{equation}
   \begin{aligned}        
   \Vert \gamma I_1 - (\mu + \delta) R_1 - \gamma I_2 + (\mu + \delta) R_2 \Vert \leq \\
       \leq \gamma \Vert I_1 - I_2 \Vert + (\mu + \delta) \Vert R_1 - R_2 \Vert.
   \end{aligned}
   \end{equation}
Then, it is clear that the local Lipschitz property will hold.

The global existence will be proved after part c).

\item[b)] If we add up the equations of \eqref{seir}, we get
\begin{align}
N'(t) = \Pi(t) - \mu N(t),
\end{align}
which has a solution in the form
\begin{equation}\label{eq:Nform}
  N(t) = N(0)\cdot e^{-\mu t} + \int_0^t\Pi(s) \cdot e^{-\mu(t-s)} \ ds.  
\end{equation}
Then,
\begin{equation}
  N(t) \leq N(0) + e^{-\mu t}\int_0^t\Pi(s) \cdot e^{\mu s} \ ds \leq N(0) + e^{-\mu t} K \int_0^t  e^{\mu s} \ ds =  
\end{equation}
\begin{equation}
    = N(0) + e^{-\mu t} K \left( \dfrac{e^{\mu t}}{\mu} - \dfrac{1}{\mu} \right) \leq N(0) + \dfrac{K}{\mu}.
\end{equation}
For the second claim, we have
\begin{equation}
    \lim_{t \rightarrow \infty} \left| N(t) - \dfrac{\Pi(t)}{\mu} \right| \leq \lim_{t \rightarrow \infty} \left| N(0)\cdot e^{-\mu t} \right|+ \lim_{t \rightarrow \infty} \left|\int_0^t\Pi(s) \cdot e^{-\mu(t-s)} \ ds - \dfrac{\Pi(t)}{\mu} \right| =  
\end{equation}
\begin{equation}
    =\lim_{t \rightarrow \infty} \left|\dfrac{\displaystyle\int_0^t\Pi(s) \cdot e^{\mu s} \ ds }{e^{\mu t}} - \dfrac{\Pi(t)}{\mu} \right|.
\end{equation}

If $\displaystyle\lim_{t \rightarrow \infty}\int_0^t\Pi(s) \cdot e^{\mu s} \ ds<\infty$, then the first term goes to zero. Moreover, \linebreak$\displaystyle\lim_{t \rightarrow \infty}\int_0^t\Pi(s) \cdot e^{\mu s} \ ds<\infty$ can only hold if $\displaystyle \lim_{t \rightarrow \infty}\Pi(t) = 0$, which gives the statement. 

On the other hand, if $\displaystyle\lim_{t \rightarrow \infty}\int_0^t\Pi(s) \cdot e^{\mu s} \ ds=\infty$, then we can use the rule of l'Hospital in the case of the first term:
\begin{equation}
    \lim_{t \rightarrow \infty} \left|\dfrac{\displaystyle\int_0^t\Pi(s) \cdot e^{\mu s} \ ds }{e^{\mu t}} - \dfrac{\Pi(t)}{\mu} \right| =  \lim_{t \rightarrow \infty} \left|\dfrac{\Pi(t) e^{\mu t}}{\mu e^{\mu t}} - \dfrac{\Pi(t)}{\mu} \right|=0.
\end{equation}

If $\mu=0$, then the previous arguments reduce to
\begin{equation}
  N(t) = N(0) + \int_0^t\Pi(s) \ ds \leq N(0) +  K t   
\end{equation}
and
\begin{equation}
\lim_{t \rightarrow \infty} N(t) = N(0) + \lim_{t \rightarrow \infty}\int_0^t\Pi(s)  ds = N(0) + \mathcal{P}. 
\end{equation}
where we used the fact that $\displaystyle\lim_{t \rightarrow \infty}\int_0^t\Pi(s)  ds = \mathcal{P}$. Note that in this case $\mathcal{P}$ can also be infinite.

\item[c)]
From the system \eqref{seir} we get the following
\begin{equation}
\begin{aligned}
\left.\frac{d S}{d t} \right|_{S = 0} &= \Pi(t), \\
\left.\frac{d E}{d t} \right|_{E = 0} &= f(I)\cdot S, \\
\left.\frac{d I}{d t} \right|_{I = 0} &= \sigma E + \delta R, \\
\left.\frac{d R}{d t} \right|_{R = 0} &= \gamma I.
\end{aligned} 
\end{equation}
Then, since all the right-hand sides of the equations are non-negative for non-negative $S, E, I$ and $R$ functions, we get  $S(t), E(t), I (t), R(t) \geq 0$ for $t > 0$ whenever $ S(0), E(0), I (0), R(0) \geq 0$ by using Lemma 1 in \cite{horvath}. 

\end{itemize}

\item[a)] (Cont'd) By part c) it is easy to see that the solutions cannot tend to infinity at a finite value of $t$, thus the global solution also exists.

\end{proof}

In the next sections, we observe how different numerical schemes applied to the model \eqref{seir} behave, meaning whether they preserve the properties of the contionuos model or not.

\section{Discrete models and their properties}\label{sec:discrete}
Since the analytic solution of model \eqref{seir} is not known, we are going to approximate these solutions by a numerical scheme.

The application of the explicit Euler method to equation \eqref{seir} with timestep $\Delta t$ results in the scheme 
\begin{equation} \label{sier_euler}
\begin{aligned}
\frac{S^{n+1} - S^n}{\Delta t} &= \Pi^n - \mu S^n - f(I^n)\cdot S^n, \\
\frac{E^{n+1} - E^n}{\Delta t} &=f(I^n)\cdot S^n - (\mu + \sigma)E^n, \\
\frac{I^{n+1} - I^n}{\Delta t} &=\sigma E^n - (\mu + \gamma)I^n + \delta R^n, \\
\frac{R^{n+1} - R^n}{\Delta t} &=\gamma I^n - (\mu + \delta) R^n
\end{aligned}
\end{equation}
with the initial conditions $S^0 = S(0)$, $E^0 = E(0)$,  $I^0 = I(0)$, $R^0 = R(0) \geq 0$ for all $n \in \mathbb{N}$.
\begin{theorem}\label{th:EE}
Assume that we apply scheme \eqref{sier_euler} to solve system \eqref{seir}. Then, the following statements hold:
\begin{itemize}
\item[a)] If
\begin{align}
\Delta t \leq \Delta t^\star = \min \left\lbrace \frac{1}{\mu + f(I^n)}, \frac{1}{\mu + \sigma}, \frac{1}{\mu + \gamma}, \frac{1}{\mu + \delta} \right\rbrace, 
\end{align}
holds for every $n=0,1,\dots$, then the explicit Euler method applied to the problem \eqref{seir} preserves the non-negativity property.
\item[b)] Let $N^k = S^k + E^k + I^k + R^k$ ($k=0,1,\dots$). If $\mu>0$ and $\Delta t \leq \dfrac{1}{\mu}$, then the estimate
\begin{equation}\label{eq:Nbound}
    N^k \leq \mathcal{N}:=N^0 +  \frac{K}{\mu}
\end{equation}
also holds for every value of $k=0,1,\dots$. Moreover, if $\Pi^k \rightarrow P$ for some $P\in \mathbb{R}$, then $\displaystyle\lim_{k \rightarrow \infty}N^k = \dfrac{P}{\mu}$. 

If $\mu=0$, then we get $N^k \leq N^0 + k \Delta t K $ and if 
\begin{equation}
    \sum_{k=0}^{\infty} \Pi^k = \widetilde{\mathcal{P}} \in \mathbb{R},
\end{equation}
then $N^k \rightarrow N^0 + \Delta t \widetilde{\mathcal{P}}$ as $k \rightarrow \infty$.
\end{itemize}
\end{theorem}

\begin{proof}\mbox{}\\*
\begin{itemize}
\item[a)] The first equation in the system \eqref{sier_euler} can be written as
\begin{equation}
S^{n+1} = S^n + \Delta t \left( \Pi^n - \mu S^n - f(I^n)\cdot S^n \right )  = S^n (1 - \Delta t\cdot\mu - \Delta t \cdot f(I^n)) + \Delta t \cdot \Pi^n.
\end{equation}
That means $S^{n+1}$ is non-negative if the condition 
\begin{align}
1 - \Delta t\cdot\mu - \Delta t \cdot f(I^n) \geq 0
\end{align}
is satisfied and we get the following upper bound for the step-size
\begin{align}
\Delta t \leq \frac{1}{\mu + f(I^n)}.
\end{align}
Similarly, by using the second equation in the system \eqref{sier_euler}
\begin{align}
E^{n+1} = E^n + \Delta t(f(I^n)\cdot S^n - (\mu + \sigma)E^n) = \Delta t \cdot f(I^n) \cdot S^n + E^n(1 - \Delta t (\mu + \sigma))
\end{align}
which is non-negative if the condition
\begin{align}
\Delta t \leq \frac{1}{\mu + \sigma}    
\end{align}
holds. From the third equation 
\begin{align}
I^{n+1} = I^n + \Delta t(\sigma E^n - (\mu + \gamma)I^n + \delta R^n) = \Delta t ((\sigma E^n + \delta R^n) + I^n(1 - \Delta t(\mu + \gamma))
\end{align}
we get the following condition for the non-negativity property
\begin{align}
\Delta t \leq \frac{1}{\mu + \gamma}.
\end{align}
Similarly from the last equation, the condition for the non-negativity preservation has the form
\begin{align}
\Delta t \leq \frac{1}{\mu + \delta}.
\end{align}

\item[b)] Let us assume that $\mu>0$. By adding up the equations of the system \eqref{sier_euler}, the sum of the four compartments has the form
\begin{equation}
\begin{aligned}
N^{n+1} &= S^{n+1} + E^{n+1} + I^{n+1} + R^{n+1} = S^n + E^n + I^n + R^n+ \\
&+ \Delta t(\Pi^n - \mu (S^n + E^n + I^n + R^n)) = N^n (1 - \Delta t \cdot \mu) + \Delta t \cdot \Pi^n.
\end{aligned}
\end{equation}
Similarly we can express $N^{n+1}$ by using $N^{n-1}$, $N^{n-2}$ and so on:
\begin{equation}
\begin{aligned}
N^{n+1} &= N^{n-1}(1 - \Delta t \cdot \mu)^2 + \Delta t \cdot \Pi^{n-1} (1 - \Delta t \cdot \mu) + \Delta t \cdot \Pi^n =\\
&= N^{n-2}(1 - \Delta t \cdot \mu)^3 + \Delta t \cdot ((1 - \Delta t \cdot \mu)^2\Pi^{n-2} + (1 - \Delta t \cdot \mu)\Pi^{n-1} + \Pi^n) =\\
&= \ldots = N^{0}(1 - \Delta t \cdot \mu)^{n+1} + \Delta t \cdot ((1 - \Delta t \cdot \mu)^n\Pi^0 + (1 - \Delta t \cdot \mu)^{n-1}\Pi^1 + \ldots + \Pi^n) \leq \\
& \leq  N^{0}(1 - \Delta t \cdot \mu)^{n+1} + \Delta t K \sum_{k = 0}^{n}(1 - \Delta t \cdot \mu)^{k}  \leq\\
&\leq N^0 + \Delta t \cdot K \cdot \frac{1-(1 - \Delta t \cdot \mu)^n}{\Delta t \cdot \mu} \leq N^0 +  \frac{K}{\mu},
\end{aligned}
\end{equation}
where we used that $\Delta t \leq \dfrac{1}{\mu}$ and $\Pi^j \leq K$ for every index $j$. 

Now we show the second claim. Since $\Pi^k \rightarrow P$ as $k \rightarrow \infty$, then for an arbitrary fixed $\varepsilon >0$ there exists an index $N_1(\varepsilon) \in \mathbb{N}$ for which for every $k>N_1(\varepsilon)$ we have $P-\varepsilon<\Pi^k<P+\varepsilon$. Now let us consider the following sequence: let us assume that $\underline{N}^n = N^n$ for $n\leq N_1(\varepsilon)$ and for $n>N_1(\varepsilon)$ we have
$$ \underline{N}^{n+1} = (1-\Delta t \mu) \underline{N}^{n} + \Delta t (P - \varepsilon). $$
It is clear that $N^n \geq \underline{N}^n$. Moreover, the sequence $\underline{N}^n$ converges: the reason for this is that the function
$$ h_1(x) = (1-\Delta t \mu) x + \Delta t (P - \varepsilon) $$
maps $\left[0, \; N(0) + \dfrac{K}{\mu}\right]$ onto itself (by the first part of part b) of this proof), and it is also a contraction on that set. Therefore, by the Banach fixed point theorem the sequence $\underline{N}^n$ converges to the fixed point of function $h_1$, which is $\dfrac{P-\varepsilon}{\mu}$. Then, the $\underline{N}^n \rightarrow\dfrac{P-\varepsilon}{\mu}$ convergence means that for the previously defined $\varepsilon>0$ there exists an index $N_2(\varepsilon)$ for which $\underline{N}^n > \dfrac{P}{\mu} - \varepsilon$ holds for every $n>N_2(\varepsilon)$. Thus, $N^n > \underline{N}^n > \dfrac{P}{\mu} - \varepsilon$ also holds for $n>\max\{ N_1(\varepsilon),N_2(\varepsilon)\}$.

Similarly, if we consider the sequence which is $\overline{N}^n = N^n$ for $n\leq N_1(\varepsilon)$ and for $n > N_1(\varepsilon)$ we have
$$ \overline{N}^{n+1} = (1-\Delta t \mu) \overline{N}^{n} + \Delta t (P + \varepsilon), $$
then $N^n \leq \overline{N}^n$ holds, and by arguments similar to the previous one, we get that it converges to $\dfrac{P+\varepsilon}{\mu}$. Then, there exists an $N_3(\varepsilon)$ for which $\overline{N}^n \leq \dfrac{P}{\mu} + \varepsilon$ holds for every $n>N_3(\varepsilon)$, meaning that $N^n < \overline{N}^n < \dfrac{P}{\mu} + \varepsilon$ also holds for $n>\max\{ N_1(\varepsilon),N_3(\varepsilon)\}$.

In conclusion, for an arbitrary $\varepsilon>0$ we showed that for \linebreak $n>\max\{ N_1(\varepsilon), N_2(\varepsilon), N_3(\varepsilon)\}$ we have $\dfrac{P}{\mu} - \varepsilon \leq N^n \leq \dfrac{P}{\mu} + \varepsilon$, which by definition means that $N^n \rightarrow \dfrac{P}{\mu}$.

In the case $\mu=0$, the previous arguments reduce to the following:
\begin{equation}
    N^{n} = N^0 + \Delta t \sum_{k=0}^{n-1} \Pi^k \leq N^0 + K n \Delta t,
\end{equation}
and if $\displaystyle\sum_{k=0}^{\infty} \Pi^n = \widetilde{\mathcal{P}}<\infty$, then
\begin{align}
\lim_{n \rightarrow \infty} N^{n} = N^0  + \lim_{n \rightarrow \infty} \Delta
t \sum_{k=0}^{n} \Pi^k = N^0  +  \Delta
t \widetilde{\mathcal{P}}.
\end{align}
Similarly, if $\displaystyle\sum_{k=0}^{\infty} \Pi^n = \widetilde{\mathcal{P}}=\infty$, then $\displaystyle\lim_{n \rightarrow \infty} N^{n}=\infty$.

\end{itemize}
\end{proof}

\begin{remark}
If $\Delta t \in (0,\Delta t^*]$, then the condition of part (b) is satisfied. 
\end{remark}

\begin{remark}
If the conditions of part (a) and (b) of Theorem \ref{th:EE} are fulfilled, then the bound \eqref{eq:Nbound} holds for $S^n$, $I^n$ and $R^n$ for every $n=0,1,2, \dots$. Therefore, if property (C4) holds for function $f$, then condition $\Delta t \leq \Delta t^*$ is satisfied if
\begin{equation}
    \Delta t \leq \min \left\lbrace \frac{1}{\mu + B}, \frac{1}{\mu + \sigma}, \frac{1}{\mu + \gamma}, \frac{1}{\mu + \delta} \right\rbrace, 
\end{equation}
where $\displaystyle B = \sup_{x \in [0,\mathcal{N}]} f(x)$. Note that this condition can be checked a priori.
\end{remark}

\begin{remark}
    Note that in the continuous case, we had $\displaystyle\lim_{t \rightarrow \infty}\left|N(t) - \dfrac{\Pi(t)}{\mu}\right|=0$, but in the discrete case we only showed this property when $\displaystyle\lim_{t \rightarrow \infty} \Pi(t)=P$. In the general case, the claim $ \displaystyle\lim_{n \rightarrow \infty} \left| N^n - \dfrac{\Pi^n}{\mu} \right|\rightarrow 0$ might not hold. 

    For example, consider the previous numerical model with $\mu=1$, $\Delta t = 1/2$, $N^0=2$ and $\Pi(t)=-\cos(2 \pi t) + 1$. Then, the numerical method has the form
    \begin{equation}
        N^{n+2} = \dfrac{1}{2} N^{n+1} + \dfrac{1}{2} \Pi^{n+1} = \dfrac{1}{2} \left( \dfrac{1}{2} N^{n} + \dfrac{1}{2} \Pi^{n} \right)  + \dfrac{1}{2} \Pi^{n+1}= \dfrac{1}{4} N^n + \dfrac{1}{4} \Pi^n + \dfrac{1}{2} \Pi^{n+1}.
    \end{equation}
\end{remark}
From the construction of $\Pi(t)$ and since $N^0=2$, it is clear that if $n$ is even (and $n\geq 1$), then $\Pi^n=0$ and $\Pi^{n+1}=2$, meaning that
\begin{equation}
    N^{n+2} = \dfrac{1}{4} N^n +1. 
\end{equation}
By Banach's fixed point theorem this converges to the value $N=\dfrac{4}{3}$. Moreover, it can also be shown that for odd values of $n$ this value is $N=\dfrac{2}{3}$. This means that the difference $\left| N^n - \dfrac{\Pi^n}{\mu} \right|$ will not converge to zero, which can be seen in Figure \ref{fig:counterex}. Thus, the convergence does not hold in the most general setting.
\begin{figure}[!h]
\centering
\includegraphics[scale=0.5]{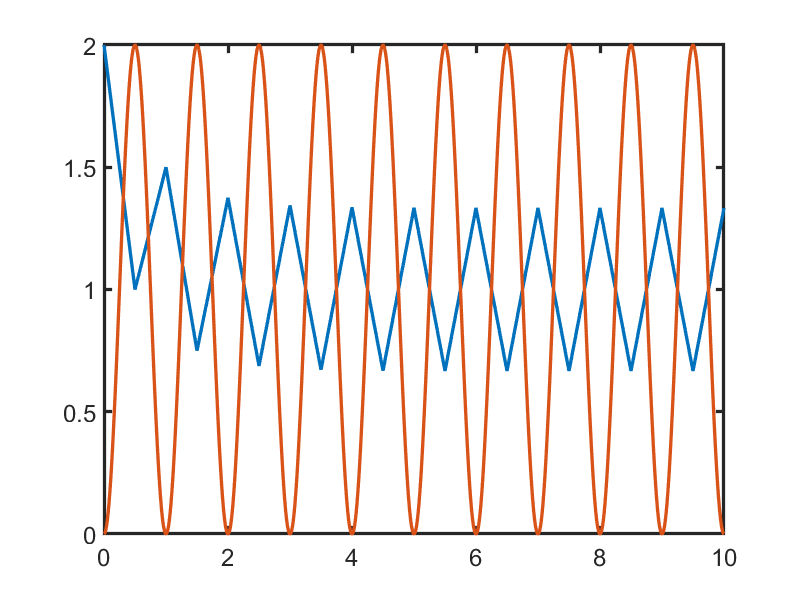}
\caption{The function $\Pi(t)$ (orange) and the numerical solution produced by the explicit Euler method for $N(t)$ (blue). As we can see, the difference of these functions does not converge to zero.}
\label{fig:counterex}
\end{figure}

\subsection{Higher order methods: strong stability preserving Runge-Kutta methods}\label{sec:RK}
Since the Euler method is only first order accurate (see Section \ref{sec:num}), we would like to apply higher-order methods to our system \eqref{seir}.

Let us consider a Runge-Kutta method in the Butcher-form \cite{butcher} with coefficients $a_{i,j}$ and $b_j$ ($i,j=1,\dots, m$). Let us also use the notations $A$ and $\boldsymbol{b}$ for the matrix and the vector containing the previous coefficients, respectively. Also, define the matrix $\mathcal{K} \in \mathbb{R}^{m \times m}$ as
\begin{align}
    \mathcal{K} = \left[ \begin{array}{cc}
        A & 0 \\
        \boldsymbol{b}^T & 0
    \end{array} \right].
\end{align}
Also, let us use the notation $\mathcal{I}$ for the $(m+1)$-dimensional identity matrix. Then, if there is such a value of $r>0$ for which $(I+r\mathcal{K})$ is invertible, then the Runge-Kutta method applied to equation
\begin{align}
    U'(t)=F(U(t))
\end{align}
can be expressed in the canonical Shu-Osher form \cite{shu1}, \cite{shu2}, \cite{shu3}
\begin{align}\label{Shu-OsherRK}
	\begin{split}
	U_{n-1}^{(i)} &= v_i U^{n-1}+\sum_{j=1}^m\alpha_{ij}\left(U_{n-1}^{(j)}+\frac{\tau}{r}F\left(U_{n-1}^{(j)}\right)\right),
    \qquad 1 \le i \le m+1, \\
    U^n &= U_{n-1}^{(m+1)},
	\end{split}
\end{align}
where the coefficients $\alpha_{ij}$ and $v_{i}$ are non-negative. These methods are usually called strong-stability preserving (or SSP) Runge-Kutta methods. 
Let us use the notations $\boldsymbol{\alpha}_r$ and $\boldsymbol{v}_r$ for the matrix and the vector collecting the aforementioned coefficients, where we emphasized the dependence on the parameter of $r$ (since different choices of $r$ might result in different schemes). 
The Shu-Osher representation with the largest value of $r$ for which $(I+r\mathcal{K})^{-1}$ exists is called optimal and the corresponding SSP coefficient is defined as
\begin{align}
	\mathcal{C} = \max\left\lbrace r \geq 0 \;|\; \exists \; (I+r\mathcal{K})^{-1} \text{ and }
		\boldsymbol{\alpha}_r \geq 0, \boldsymbol{v}_r \geq 0 \right\rbrace.
\end{align}
Then, an explicit Runge-Kutta method given in the optimal Shu-Osher form applied to equation \eqref{seir} has the form
\begin{subequations}\label{eq:SSPRK}
		\begin{align}
			S_{n-1}^{(i)} &= v_i S^{n-1} + \sum_{j=1}^{i-1} \alpha_{ij}\left(S_{n-1}^{(j)} +
			\frac{\tau}{\mathcal{C}}\left(\Pi_{n-1}^{(j)} - \mu S_{n-1}^{(j)} -f(I_{n-1}^{(j)}) \cdot S_{n-1}^{(j)}\right)\right), \label{ith_stageS}\\
            E_{n-1}^{(i)} &= v_i E^{n-1} + \sum_{j=1}^{i-1} \alpha_{ij}\left(E_{n-1}^{(j)} +
			\frac{\tau}{\mathcal{C}}\left(f(I_{n-1}^{(j)}) \cdot S_{n-1}^{(j)} - (\mu + \sigma) E_{n-1}^{(j)}\right)\right), \label{ith_stageE}\\
			I_{n-1}^{(i)} &= v_i I^{n-1} + \sum_{j=1}^{i-1} \alpha_{ij}\left(I_{n-1}^{(j)} +
			\frac{\tau}{\mathcal{C}}\left(\sigma E_{n-1}^{(j)} - (\mu + \gamma)I_{n-1}^{(j)} + \delta R_{n-1}^{(j)}\right)\right),\label{ith_stageI}\\
			R_{n-1}^{(i)} &= v_i R^{n-1} + \sum_{j=1}^{i-1} \alpha_{ij}\left(R_{n-1}^{(j)} +
			\frac{\tau}{\mathcal{C}}\left(\gamma I_{n-1}^{(j)} - (\mu+\delta) R_{n-1}^{(j)}\right)\right), \label{ith_stageR}\\
            S^{n} &= S_{n-1}^{(m+1)}, \\
            E^{n} &= E_{n-1}^{(m+1)}, \\
            I^{n} &= I_{n-1}^{(m+1)}, \\
            R^{n} &= R_{n-1}^{(m+1)}, 
		\end{align}
	\end{subequations}
where $1 \le i \le m+1$.

The next theorem gives a sufficient condition for the timesteps under which the non-negativity property holds.

\begin{theorem}\label{th:RK}
    Let us apply an explicit Runge-Kutta scheme given in the form \eqref{eq:SSPRK}. 
    \begin{itemize}
        \item[a)] If the time-step satisfies
    \begin{align}\label{eq:RK_cond}
        0 < \tau \leq \mathcal{C} \Delta t^*
    \end{align}
    (where $\Delta t^*$ is the same as defined in Theorem \ref{th:EE}), then the non-negativity property holds.
        \item[b)] Let $N^k = S^k +E^k+ I^k + R^k$ ($k=0,1,\dots$). If $\mu>0$ and $\tau \leq \mathcal{C} \dfrac{1}{\mu}$, then the estimate
\begin{equation}\label{eq:Nbound_RK}
    N^k \leq \mathcal{N}:=N^0 +  \frac{K}{\mu}
\end{equation}
also holds for every value of $k=0,1,\dots$. Moreover, if $\displaystyle\lim_{k\rightarrow \infty}\Pi^k = P$, then $N^k \rightarrow \dfrac{P}{\mu}$ as $k \rightarrow \infty$.

Additionally, in the case $\mu=0$ we have $N^n \leq N^0 + n \dfrac{\tau}{\mathcal{C}}K$ for every $n$. Furthermore, if $\displaystyle\sum_{n=0}^{\infty}\sum_{j=1}^m \Pi_n^{(j)}< \infty$, then $ N^n $ converges to a finite positive value as $n \rightarrow \infty$, and if $\displaystyle\lim_{n \rightarrow \infty}\Pi^n = P>0$, then $\displaystyle\lim_{n \rightarrow \infty}N^n =\infty$.
    \end{itemize} 
\end{theorem}

\begin{proof}
We prove the statements by induction.
\begin{itemize}
    \item[a)] Let us assume that the non-negative property holds for $S^k$ ($k=1,2,\dots n-1$). Then, equation \eqref{ith_stageS} can be rewritten as
    \begin{align}
        	S_{n-1}^{(i)} = v_i S^{n-1} + \sum_{j=1}^{i-1} \alpha_{ij}\left(\frac{\tau}{\mathcal{C}}\Pi_{n-1}^{(j)} + S_{n-1}^{(j)}\left( 1 +
			 - \frac{\tau}{\mathcal{C}} \mu  - \frac{\tau}{\mathcal{C}}f(I_{n-1}^{(j)}) \right)\right)
    \end{align}
Then $S_{n-1}^{(1)} \geq 0$ holds, and if $\tau \leq \mathcal{C} \Delta t^*$, then since $\tau \leq \dfrac{\mathcal{C}}{\mu + B}\leq \dfrac{\mathcal{C}}{\mu + f(I_{n-1}^{(j)})}$, the non-negativity property holds also for $i=k$ if it was true for $i=1,2,\dots k-1$. Moreover, $f(I_{n-1}^{(j)}) \leq B$. Therefore, the non-negativity holds for $S$.

Now let us assume that the non-negative property holds for $E^k$ ($k=1,2,\dots n-1$). We can rewrite equation \eqref{ith_stageE} similarly as \eqref{ith_stageS}:
\begin{align}
        	E_{n-1}^{(i)} = v_i E^{n-1} + \sum_{j=1}^{i-1} \alpha_{ij}\left( 
			\frac{\tau}{\mathcal{C}}f(I_{n-1}^{(j)}) \cdot S_{n-1}^{(j)} + E_{n-1}^{(j)}\left( 1 - 
			\frac{\tau}{\mathcal{C}}(\mu + \sigma) \right)\right)
    \end{align}    
Therefore, $S_{n-1}^{(1)} \geq 0$ holds. Also, it is clear that the term $f(I_{n-1}^{(j)}) \cdot S_{n-1}^{(j)}$ is non-negative, and if $\tau \leq \mathcal{C} \Delta t^*$, then since $\tau \leq \dfrac{\mathcal{C}}{\mu + \sigma}$, the non-negativity property holds also for $i=k$ if it was true for $i=1,2,\dots k-1$. Therefore, the non-negativity holds for $E$.

Now we prove the property for $I$ and $R$. Let us assume that the non-negative property holds for $I^k$ and $R^k$ ($k=1,2,\dots n-1$). We can rewrite equations \eqref{ith_stageI} and \eqref{ith_stageR} as
\begin{equation}
    \begin{aligned}
        I_{n-1}^{(i)} &= v_i I^{n-1} + \sum_{j=1}^{i-1} \alpha_{ij}\left(
			\frac{\tau}{\mathcal{C}}\sigma E_{n-1}^{(j)} + \left( 1 - \frac{\tau}{\mathcal{C}}(\mu + \gamma)\right)I_{n-1}^{(j)} + \frac{\tau}{\mathcal{C}}\delta R_{n-1}^{(j)}\right),\\
			R_{n-1}^{(i)} &= v_i R_{n-1}^{n-1} + \sum_{j=1}^{i-1} \alpha_{ij}\left(
			\gamma \frac{\tau}{\mathcal{C}} I_{n-1}^{(j)} + \left(1- \frac{\tau}{\mathcal{C}} (\mu+\delta) \right)R_{n-1}^{(j)}\right).
    \end{aligned}
\end{equation}
It is easy to see that $I_{n-1}^{(1)}, R_{n-1}^{(1)}\geq 0$. Moreover, we know that $E_{n-1}^{(j)}\geq 0$ ($j=1,2,\dots, m+1$) by the previous claims. Now let us assume that $I_{n-1}^{(k)}\geq 0$ and $R_{n-1}^{(k)}\geq 0$ for $k=1,\dots, i-1$. If $\tau\leq \mathcal{C} \Delta t^*$, then $\tau \leq \dfrac{\mathcal{C}}{\mu+\gamma}$ and $\tau \leq \dfrac{\mathcal{C}}{\mu+\delta}$ is also true, meaning that the non-negative property also holds for $I_{n-1}^{(i)}$ and $R_{n-1}^{(i)}$. This concludes this part of the proof.

\item[b)] Let us assume that $\mu>0$. If we add up equations \eqref{ith_stageS}--\eqref{ith_stageR}, then we get
\begin{equation}
\begin{aligned}
    N_{n-1}^{(i)} =& S_{n-1}^{(i)} + E_{n-1}^{(i)} + I_{n-1}^{(i)} + R_{n-1}^{(i)} = \\
    = & v_i (S^{n-1}+E^{n-1} + I^{n-1} + R^{n-1})  +\\ & + \sum_{j=1}^{i-1}  \alpha_{ij} \left( S_{n-1}^{(j)} + E_{n-1}^{(j)} + I_{n-1}^{(j)} + R_{n-1}^{(j)} + \dfrac{\tau}{\mathcal{C}}\Pi_{n-1}^{(j)} - \dfrac{\tau}{\mathcal{C}} \mu (S_{n-1}^{(j)} + E_{n-1}^{(j)} + I_{n-1}^{(j)} + R_{n-1}^{(j)})  \right)=\\
    = & v_i N^{n-1} + \dfrac{\tau}{\mathcal{C}}\sum_{j=1}^{i-1} \alpha_{ij} \Pi_{n-1}^{(j)} + \sum_{j=1}^{i-1} \alpha_{ij} \left( N_{n-1}^{(j)} - \dfrac{\tau}{\mathcal{C}} \mu N_{n-1}^{(j)}   \right) = \\
    = & v_i N^{n-1} + \dfrac{\tau}{\mathcal{C}}\sum_{j=1}^{i-1} \alpha_{ij} \Pi_{n-1}^{(j)} + \left( 1 - \dfrac{\tau}{\mathcal{C}} \mu \right) \sum_{j=1}^{i-1} \alpha_{ij} N_{n-1}^{(j)} \leq \\
    \leq & v_i N^{n-1} + \dfrac{\tau}{\mathcal{C}} K \sum_{j=1}^{i-1} \alpha_{ij} + \left( 1 - \dfrac{\tau}{\mathcal{C}} \mu \right) \sum_{j=1}^{i-1} \alpha_{ij} N_{n-1}^{(j)}.
\end{aligned}
\end{equation}
Consistency requires that $\displaystyle v_i + \sum_{j=1}^{i-1} \alpha_{ij} = 1$, meaning that we have
\begin{align}
    N_{n-1}^{(i)} \leq v_i N^{n-1} + \dfrac{\tau}{\mathcal{C}}K(1-v_i) + \left( 1 - \dfrac{\tau}{\mathcal{C}} \mu \right) \sum_{j=1}^{i-1} \alpha_{ij} N_{n-1}^{(j)}.
\end{align}

Consistency also means that $v_1=1$, $0\leq v_i \leq 1$ and $0 \leq \alpha_{ij} \leq 1$. 

For the proof, we will use the following two lemmas (their proofs can be found in the Appendix):
\begin{lemma}\label{lem:shupresproof1}
    If we have a method of $m$ steps, then 
    $$N^{n} = N_{n-1}^{(m+1)} \leq N^{n-1} A_{m+1} + \dfrac{\tau}{\mathcal{C}} K B_{m+1},$$
where 
$$A_1 = 1, \qquad A_i = v_{i} + \left( 1 - \dfrac{\tau}{\mathcal{C}} \mu \right) \sum_{j=1}^{i-1} \alpha_{ij} A_j \quad (i=2,3,\dots,m), $$
and 
$$B_1=0, \qquad B_i = 1-v_i + \left( 1 - \dfrac{\tau}{\mathcal{C}} \mu \right) \sum_{j=1}^{i-1} \alpha_{ij} B_j \quad (i=2,3,\dots,m),$$
where $0 \leq A_i \leq 1$ and $0 \leq B_i \leq 1$ for $i=1,2,\dots m$.
\end{lemma}

\begin{lemma}\label{lem:shupresproof2}
    $\dfrac{\tau}{\mathcal{C}}\mu B_i = 1-A_i $ holds for every $i=1,2,\dots m$.
\end{lemma}

Then, if we have an $m$-step method, by Lemma \ref{lem:shupresproof1} the steps are
$$N^{n+1} \leq A_{m+1} N^n + \dfrac{\tau}{\mathcal{C}}K B_{m+1} \leq$$
$$\leq A_{m+1}^{n+1} N^0 + \dfrac{\tau}{\mathcal{C}}K B_m (\left(A_{m+1}\right)^n + \left(A_{m+1}\right)^{n-1} + \dots + \left(A_{m+1}\right)^{2} + A_{m+1} + 1) =$$
$$ =\left(A_{m+1}\right)^{n+1} N^0 + \dfrac{\tau}{\mathcal{C}}K B_{m+1} \dfrac{1-\left(A_{m+1}\right)^{n+1}}{1-A_{m+1}} = $$
By Lemma \ref{lem:shupresproof2} we get
$$ =\left(A_{m+1}\right)^{n+1} N^0 + \dfrac{\tau}{\mathcal{C}}K B_{m+1} \dfrac{1-\left(A_{m+1}\right)^{n+1}}{\dfrac{\tau}{\mathcal{C}}\mu B_{m+1}} = \left(A_{m+1}\right)^{n+1} N^0 + \dfrac{K}{\mu}\left(1-\left(A_{m+1}\right)^{n+1}\right).$$
By Lemma \ref{lem:shupresproof1} we have
$$ N^{n+1} \leq N^0 + \dfrac{K}{\mu}.$$

For the second claim, we can use similar arguments as the proof of Theorem \ref{th:EE}. Namely, if $\Pi^k \rightarrow P$, then for any arbitrary fixed value of $\varepsilon$ there exists an $N_1(\varepsilon)$ index for which for every $n \geq N_1(\varepsilon)$ we have $P-\varepsilon \leq N^n \leq P+\varepsilon$.

Now let us consider the following numerical scheme: $\underline{N}^n=N^n$ for $n\leq N_1(\varepsilon)$ and for $n>N_1(\varepsilon)$ we have
$$ \underline{N}^n = \underline{N}^{n-1} A_{m+1} + \dfrac{\tau}{\mathcal{C}} (P-\varepsilon) B_{m+1}, $$
where $A_{m+1}$, $B_{m+1}$ are as defined before. It can be seen by induction that $\underline{N}^n \leq N^n$. Moreover, the sequence $\underline{N}^n$ converges, since the function
$$ h_2(x) = x A_{m+1} + \dfrac{\tau}{\mathcal{C}} (P-\varepsilon) B_{m+1} $$
is a contraction, and it maps the interval $\left[0, N^0+\dfrac{K}{\mu} \right]$ onto itself, and then we can use Banach's fixed point theorem. The fixed point of this function is
$$ x_0 = \dfrac{\dfrac{\tau}{\mathcal{C}} (P-\varepsilon) B_{m+1}}{1-A_{m+1}} = \dfrac{P-\varepsilon}{\mu}, $$
where we used Lemma \ref{lem:shupresproof2}. This means that $\underline{N}^n \rightarrow \dfrac{P-\varepsilon}{\mu}$, meaning that for the previous fixed value of $\varepsilon$ there exists an index $N_2(\varepsilon)$ for which $\underline{N}^n \geq \dfrac{P}{\mu} - \varepsilon$. Thus, $N^n > \underline{N}^n > \dfrac{P}{\mu}-\varepsilon$ holds for $n > \max\{ N_1(\varepsilon), N_2(\varepsilon) \}$.

By a similar argument, if we consider the sequence $\overline{N}^n$ for which we have $\overline{N}^n = N^n$ for $n\leq N_1(\varepsilon)$ and for $n>N_1(\varepsilon)$
$$ \overline{N}^n = \overline{N}^{n-1} A_{m+1} + \dfrac{\tau}{\mathcal{C}} (P+\varepsilon) B_{m+1}, $$
then $\overline{N}^n \geq N^n$ and since $\overline{N}^n \rightarrow \dfrac{P + \varepsilon}{\mu}$, then there exists an index $N_3(\varepsilon)$ for which $N^n \leq \overline{N}^n \leq \dfrac{P}{\mu} + \varepsilon$. Then, since for $n>\max\{ N_1(\varepsilon), N_2(\varepsilon), N_3(\varepsilon) \}$ we have $\dfrac{P}{\mu} - \varepsilon < N^n < \dfrac{P}{\mu} + \varepsilon$, by definition $N^n \rightarrow \dfrac{P}{\mu}$.\medskip

In the $\mu=0$ case, the result of Lemma \ref{lem:shupresproof1} reduces to
\begin{align}
    N^n = N^{(m)} \leq N^{n-1} \widetilde{A}_{m+1} + \dfrac{\tau}{\mathcal{C}} K \widetilde{B}_{m+1},
\end{align}
where since $\widetilde{A}_1=1$ and $\displaystyle \widetilde{A}_i = v_i + \sum_{j=1}^{i-1} \alpha_{ij} \widetilde{A}_j$, by using $\displaystyle v_i + \sum_{j=1}^{i-1} \alpha_{ij}=1$ we get that $\widetilde{A}_i=1$ for every $j$ index. Similarly, we have the $\widetilde{B}_1=0$ and $\displaystyle \widetilde{B}_i = 1-v_i + \sum_{j=1}^{i-1} \alpha_{ij} \widetilde{B}_j$, and $0 \leq \widetilde{B}_i \leq 1$ also holds for every value of $j$. Then,
$$N^{n} \leq  N^{n-1} + \dfrac{\tau}{\mathcal{C}}K \widetilde{B}_{m+1} \leq N^0 + n \dfrac{\tau}{\mathcal{C}}K \widetilde{B}_{m+1}  \leq N^0 + n \dfrac{\tau}{\mathcal{C}}K .$$

Moreover, in this case
$$ N^{n} = N_{n-1}^{(m+1)} = v_{m+1} N^{n-1} + \dfrac{\tau}{\mathcal{C}} \sum_{j=1}^{m} \alpha_{m+1,j} \Pi_{n-1}^{(j)} + \sum_{j=1}^{m} \alpha_{m+1,j} N_{n-1}^{(j)} $$
\begin{lemma}\label{lem:shupresproof3}
    $$N_{n-1}^{(i)} = N^{n-1} + \dfrac{\tau}{\mathcal{C}} \sum_{j=1}^{i-1} \gamma_{ij} \Pi_{n-1}^{(j)}, $$
    where
    $$\gamma_{ij} = \alpha_{ij} + \sum_{k=j+1}^{i-1} \alpha_{ik} \gamma_{kj}. $$
\end{lemma}
(The proof of this claim can be found in the Appendix.)

Therefore, if $\displaystyle\sum_{n=0}^{\infty} \sum_{j=1}^{m+1} \Pi_{n}^{(j)}<\infty$, then
$$ \lim_{n \rightarrow \infty} N^n = N^0 + \dfrac{\tau}{\mathcal{C}} \sum_{n=0}^{\infty} \sum_{j=1}^{m+1} \gamma_{mj} \Pi_{n}^{(j)} \leq N^0 + \dfrac{\tau}{\mathcal{C}}\sum_{n=0}^{\infty} \sum_{j=1}^{m+1} \Pi_{n}^{(j)} <\infty. $$
Here we used that $0 \leq \gamma_{ij}\leq 1$, which can be seen by induction: for the first few elements it is true, and then if $\gamma_{ij}\leq 1$ holds for $i< k$, $j<\ell$, then (here $k<\ell$)
$$ \gamma_{k \ell} = \alpha_{k \ell} + \sum_{p=\ell+1}^{k-1} \alpha_{k p} \gamma_{p \ell} \leq \alpha_{k \ell} + \sum_{p=\ell+1}^{k-1} \alpha_{k p} \leq 1 - v_k \leq 1. $$

Furthermore, if $\Pi^n \rightarrow P>0$ as $n \rightarrow \infty$, then for a fixed value of $0<\varepsilon<P$ there exists an index $N_4(\varepsilon)$ for which for every $n>N_4(\varepsilon)$ we have \linebreak $\Pi^n > P-\varepsilon$. Then, we have
$$ \lim_{n \rightarrow \infty} N^n = N^0 + \dfrac{\tau}{\mathcal{C}} \sum_{n=0}^{\infty} \sum_{j=1}^{m} \gamma_{mj} \Pi_{n}^{(j)} =$$
$$ = N^0 + \dfrac{\tau}{\mathcal{C}} \sum_{n=0}^{N_4(\varepsilon)} \sum_{j=1}^{m} \gamma_{mj} \Pi_{n}^{(j)} + \dfrac{\tau}{\mathcal{C}} \sum_{n=N_4(\varepsilon) + 1}^{\infty} \sum_{j=1}^{m} \gamma_{mj} \Pi_{n}^{(j)} > $$
$$ > N^0 + \dfrac{\tau}{\mathcal{C}} \sum_{n=0}^{N_4(\varepsilon)} \sum_{j=1}^{m} \gamma_{mj} \Pi_{n}^{(j)} + \dfrac{\tau}{\mathcal{C}} \sum_{n=N_4(\varepsilon) + 1}^{\infty} \sum_{j=1}^{m} \gamma_{mj} (P-\varepsilon) =\infty, $$
where in the last step we used the fact that $\displaystyle\sum_{j=1}^{m} \gamma_{mj} (P-\varepsilon)$ is a positive constant. 
\end{itemize}

This concludes the proof of Theorem \ref{th:RK}.
\end{proof}

\begin{remark}
    If our Runge-Kutta method is explicit, then $v_i=1$ holds, meaning that the first stage of the method is always $X^{(1)}=X^n$ for every $X \in \{ S,I,E,R\}$. From now on we will omit this first stage, and shift every index of the stages accordingly, meaning that $\widetilde{S^{(1)}}:=S^{(2)}$, $\widetilde{S^{(2)}}:=S^{(3)}$ etc. From now on we will also omit the use of these tildes.
\end{remark}

\section{Numerical experiments}\label{sec:num}

In this section, we conduct some numerical experiments and see whether the theoretical results of the previous sections are confirmed by these.

For our numerical tests, we are using the parameters $\mu=0.05$, $\sigma=0.25$, $\gamma=0.1867$, $\delta=0.011$, final time $t_f = 1000$ and the incidence function
$$g(x) = 0.0115 \cdot e^{-0.001 x}$$
(for details, see \cite{wang}). For the recruitment rate function $\Pi(t)$ \cite{pelda_recruit}, we apply three different choices:
\begin{itemize}
    \item[A)] $\Pi(t) = \kappa \left( \dfrac{2}{\pi}\mathrm{arctg}(t) + \dfrac{\sin(t)}{t} \right)$,
    \item[B)] $\Pi(t) = \kappa \left( \dfrac{1}{\pi}\mathrm{arctg}(t) + \dfrac{1}{2} \right)$,
    \item[C)] $\Pi(t) = \kappa \left( -t e^{-t} + 1 \right)$,
\end{itemize}
where $\kappa \in \mathbb{R}$. We are going to apply the choice $\kappa=\mu=0.05$ (so that \linebreak $\displaystyle\lim_{t \rightarrow \infty} \Pi(t) = \mu$). The initial conditions used in the tests are $S(0)=0.2$, $E(0)=0.6$, $I(0)=0.2$ and $R(0)=0$.

In Table \ref{bounds_table_a} we can see the sufficient bounds given by Theorems \ref{th:EE} and \ref{th:RK} (denoted here by $\tau_t$) and the 'real'  bounds given by numerical experiments (denoted by $\tau_r$), i.e. the values of $\tau$ above which the method will not preserve the desired properties. As we can see, the theoretical bounds are relatively close to the real, necessary bounds in most cases. 

\begin{table}
\centering
    \small
    \begin{tabular*}{\textwidth}{@{\extracolsep{\fill}}c*{5}{|c@{\hspace{7pt}}c@{\hspace{7pt}}c}}
		\toprule
		\multicolumn{1}{c}{$\Pi(t)$} & 
        \multicolumn{1}{c}{$\tau_t$ (EE)} &
        \multicolumn{1}{c}{$\tau_r$ (EE)} &
		\multicolumn{1}{c}{$\tau_r/\tau_t$} & 
        \multicolumn{1}{c}{$\tau_t$ (RK2)} &
        \multicolumn{1}{c}{$\tau_r$ (RK2)} &
		\multicolumn{1}{c}{$\tau_r/\tau_t$} & 
  \\
		\midrule
A) & $3.3333$ & $3.5223$ & $1.0567$ & $3.3333$ & $4.5688$ & $1.3706$ \\ 
B) & $3.3333$ & $3.5223$ & $1.0567$ & $3.3333$ & $4.5697$ & $1.3709$ \\ 
C) & $3.3333$ & $3.5223$ & $1.0567$ & $3.3333$ & $4.5678$ & $1.3703$ \\ 	
		\bottomrule
	\end{tabular*}

    \begin{tabular*}{\textwidth}{@{\extracolsep{\fill}}c*{5}{|c@{\hspace{7pt}}c@{\hspace{7pt}}c}}
		\toprule
		\multicolumn{1}{c}{$\Pi(t)$} & 
        \multicolumn{1}{c}{$\tau_t$ (RK3)} &
        \multicolumn{1}{c}{$\tau_r$ (RK3)} &
		\multicolumn{1}{c}{$\tau_r/\tau_t$} & 
        \multicolumn{1}{c}{$\tau_t$ (RK4)} &
        \multicolumn{1}{c}{$\tau_r$ (RK4)} &
		\multicolumn{1}{c}{$\tau_r/\tau_t$} & 
  \\
		\midrule
A)  & $3.3333$ & $5.5164$ & $1.6549$ & $20.0000$ & $29.8408$ & $1.4920$\\ 
B) & $3.3333$ & $5.5167$ & $1.6550$ & $20.0000$ & $29.8648$ & $1.4932$ \\ 
C) & $3.3333$ & $5.5203$ & $1.6561$ & $20.0000$ & $29.7721$ & $1.4886$ \\ 
		\bottomrule
	\end{tabular*}
 
		\caption{The bounds for the different order methods for different choices of $\Pi(t)$, with parameters $\mu=0.05$, $\sigma=0.25$, $\gamma=0.1867$, $\delta=0.011$, $g(x)=0.0115 \cdot e^{-0.001x}$
		 and final time $t_f=1000$.}
    \label{bounds_table_a}
\end{table}

In Figure \ref{fig:bad-timestep} we also observed the behavior of the method when we use time-steps larger than the bound $\tau_r$. For convenience, here we use final time $t_f=30$, and we also use the choice C) for the function $\Pi(t)$ along with a second order method. In this case, the bound we get from the theorem is $\tau_t = 3.3333$. First, we used a time-step that was bigger, namely $\tau=4.8$. As we can see, function $I$ gets negative, thus violating the property we would like to preserve. However, if we use a time-step below the bound, namely $\tau=3.3$, then the method behaves as expected.

\begin{figure}
    \centering
    \includegraphics[scale=0.285]{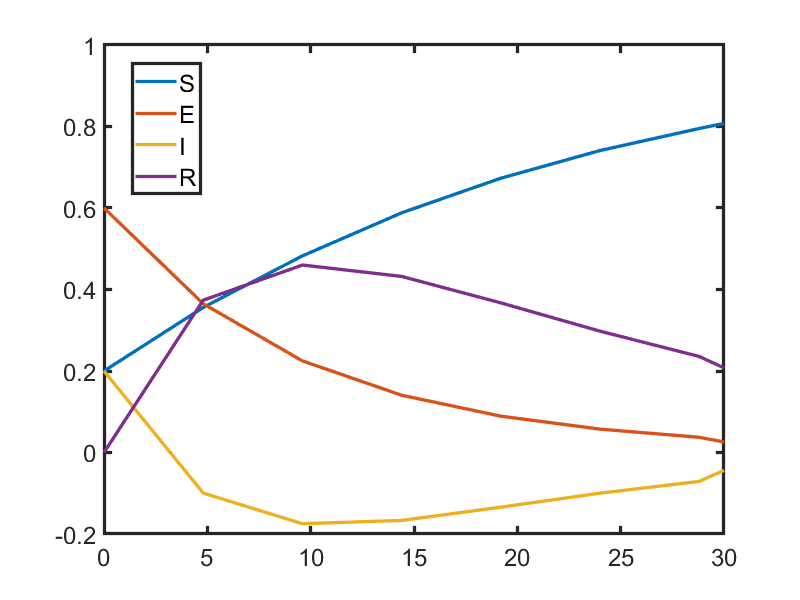}
    \includegraphics[scale=0.285]{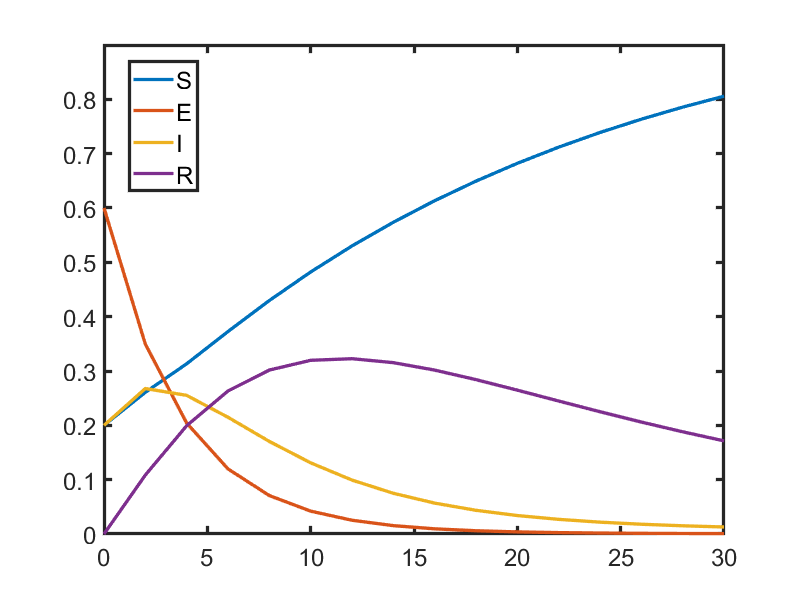}
    \caption{The second order method with timestep $\tau=4.8$ (left) and $\tau=3.3$ (right). The method violates the property on the left figure but behaves as expected on the right one.}
    \label{fig:bad-timestep}
\end{figure}

In Figure \ref{fig:conv-order}, we examined the order of the different methods with choice A) for the function $\Pi(t)$. For this, we calculated the solution by using the fourth-order method and a very small step size ($\tau=\tau_t \cdot 2^{-10}$ where $\tau_t$ is the bound given by Theorem \ref{th:RK}). Then, this solution of high precision is regarded as a reference solution, and then we observe the difference between this reference solution and the solutions we get by using different step sizes. As we can see, the methods behave as expected, i.e. the forward Euler method is of first, the SSPRK2 method is of second order, and so on.

\begin{figure}[!h]
    \centering
    \includegraphics[scale=0.4]{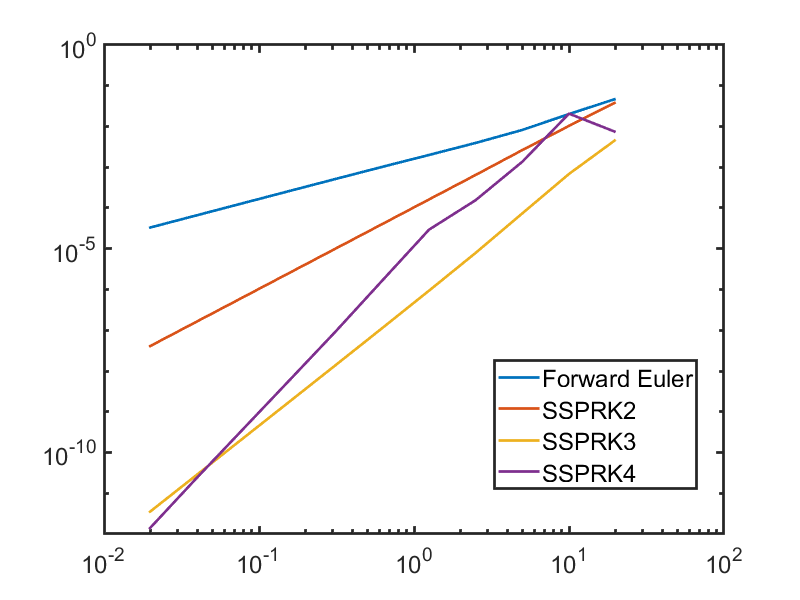}
    \caption{The orders of the different methods. As we can see, they behave as expected, i.e. the second order method attains an order of two, and so on.}
    \label{fig:conv-order}
\end{figure}

\section*{Appendix}
Here we list the proofs of Lemmas \ref{lem:shupresproof1}, \ref{lem:shupresproof2} and \ref{lem:shupresproof3}.

\begin{proof}[of Lemma \ref{lem:shupresproof1}]
First we prove the first claim. It can be seen that it holds for the first few steps, and then by induction,
$$\begin{aligned}
    N_{n-1}^{(i)} & \leq v_i N^{n-1} + \dfrac{\tau}{\mathcal{C}}K(1-v_i) + \left( 1 - \dfrac{\tau}{\mathcal{C}} \mu \right) \sum_{j=1}^{i-1} \alpha_{ij} N_{n-1}^{(j)}\leq\\
    & \leq v_i N^{n-1} + \dfrac{\tau}{\mathcal{C}}K(1-v_i) + \left( 1 - \dfrac{\tau}{\mathcal{C}} \mu \right) \sum_{j=1}^{i-1} \alpha_{ij} \left( N^{n-1} A_{j} + \dfrac{\tau}{\mathcal{C}} K B_{j} \right) =\\
    & = N^{n-1} \left( v_i + \left( 1 - \dfrac{\tau}{\mathcal{C}} \mu \right) \sum_{j=1}^{i-1} \alpha_{ij} A_j \right) + \dfrac{\tau}{\mathcal{C}}K\left(1-v_i + \left( 1 - \dfrac{\tau}{\mathcal{C}} \mu \right) \sum_{j=1}^{i-1} \alpha_{ij} B_j \right) =\\
    & = N^{n-1} A_i + \dfrac{\tau}{\mathcal{C}}K B_i.
\end{aligned}$$
Now we prove the property of $A_i$. By induction: $0 \leq A_1 \leq 1$ and if it holds for $A_j$, $j=1,2,\dots, i-1$, then
    $$0 \leq A_i = v_{i} + \left( 1 - \dfrac{\tau}{\mathcal{C}} \mu \right) \left(\alpha_{i1}A_1  + \alpha_{i2}A_2 + \dots + \alpha_{i,i-1} A_{i-1} \right) \leq v_{i} + \alpha_{i1}  + \alpha_{i2} + \dots + \alpha_{i,i-1} = 1$$
The property can be proved similarly for $B_i$ by induction: $0 \leq B_1=0 \leq 1$, $0\leq B_2 = 1-v_2\leq 1$ and if it holds for $B_j$, $j=1,2,\dots, i-1$, then
    $$0 \leq B_i = 1-v_i + \left( 1 - \dfrac{\tau}{\mathcal{C}} \mu \right) \left(\alpha_{i2}B_2  + \alpha_{i3}B_3 + \dots + \alpha_{i,i-1} B_{i-1}\right) \leq $$
    $$ \leq 1-v_i + \alpha_{i2}  + \alpha_{i3}B_3 + \dots + \alpha_{i,i-1} B_{i-1} =1-v_i + v_i = 1.  $$
\end{proof}

\begin{proof}[of Lemma \ref{lem:shupresproof2}]
    We prove the statement by induction. If $i=1$, then the statement is trivial. Moreover, for $i=2$ we have
    $$ \dfrac{\tau}{\mathcal{C}}\mu B_2 = \dfrac{\tau}{\mathcal{C}}\mu (1-v_2) = \dfrac{\tau}{\mathcal{C}}\mu \alpha_{21}  $$
    and similarly
    $$ 1-A_2 = 1- \left( v_2 + \left( 1-\dfrac{\tau}{\mathcal{C}}\mu \right) \alpha_{21} \right) = \alpha_{21} - \left( 1-\dfrac{\tau}{\mathcal{C}}\mu \right) \alpha_{21} = \dfrac{\tau}{\mathcal{C}}\mu\alpha_{21}. $$
Now let us assume that the statement holds for $j=1,2,\dots i-1$ and we would like to prove it for $j=i$. Then,
$$ \dfrac{\tau}{\mathcal{C}}\mu B_i =  \dfrac{\tau}{\mathcal{C}}\mu (1-v_i) + \dfrac{\tau}{\mathcal{C}}\mu\left( 1 - \dfrac{\tau}{\mathcal{C}} \mu \right) \sum_{j=1}^{i-1} \alpha_{ij} B_j = $$
$$ = \dfrac{\tau}{\mathcal{C}}\mu \left(\sum_{j=1}^{i-1} \alpha_{ij}\right) + \left( 1 - \dfrac{\tau}{\mathcal{C}} \mu \right) \sum_{j=1}^{i-1} \alpha_{ij} \dfrac{\tau}{\mathcal{C}}\mu B_j =$$
    Now we use the assumption that the statement holds for $j=1,2,\dots i-1$ meaning that $\dfrac{\tau}{\mathcal{C}}\mu B_j = 1- A_j$ for $j=1,2,\dots i-1$:
 $$ = \dfrac{\tau}{\mathcal{C}}\mu \left(\sum_{j=1}^{i-1} \alpha_{ij}\right) + \left( 1 - \dfrac{\tau}{\mathcal{C}} \mu \right) \sum_{j=1}^{i-1} \alpha_{ij} (1-A_j) = $$
 $$= \dfrac{\tau}{\mathcal{C}}\mu \left(\sum_{j=1}^{i-1} \alpha_{ij}\right) + \left( 1 - \dfrac{\tau}{\mathcal{C}} \mu \right) \left(\sum_{j=1}^{i-1} \alpha_{ij}  - \sum_{j=1}^{i-1} \alpha_{ij} A_j \right)=  $$   
$$= \dfrac{\tau}{\mathcal{C}}\mu \left(\sum_{j=1}^{i-1} \alpha_{ij}\right) + \sum_{j=1}^{i-1} \alpha_{ij}  - \sum_{j=1}^{i-1} \alpha_{ij} A_j - \dfrac{\tau}{\mathcal{C}}\mu \sum_{j=1}^{i-1} \alpha_{ij} + \dfrac{\tau}{\mathcal{C}}\mu \sum_{j=1}^{i-1} \alpha_{ij} A_j =  $$
$$ =  \sum_{j=1}^{i-1} \alpha_{ij} -\left( 1 - \dfrac{\tau}{\mathcal{C}} \mu \right) \sum_{j=1}^{i-1} \alpha_{ij} A_j.$$
Similarly, for the right-hand side we have
$$ 1 - A_i = 1-\left(  v_{i} + \left( 1 - \dfrac{\tau}{\mathcal{C}} \mu \right) \sum_{j=1}^{i-1} \alpha_{ij} A_j  \right) = 1-v_i - \left( 1 - \dfrac{\tau}{\mathcal{C}} \mu \right) \sum_{j=1}^{i-1} \alpha_{ij} A_j = $$
$$ = \sum_{j=1}^{i-1} \alpha_{ij} - \left( 1 - \dfrac{\tau}{\mathcal{C}} \mu \right) \sum_{j=1}^{i-1} \alpha_{ij} A_j, $$
which gives the statement of the lemma.
\end{proof}

\begin{proof}[of Lemma \ref{lem:shupresproof3}]
    We prove the statement by induction in $i$. It is clear that $N_{n-1}^{(1)}=N^{n-1}$ and
    $$ N_{n-1}^{(2)} = v_2 N^{n-1} + \dfrac{\tau}{\mathcal{C}} \alpha_{21} \Pi_{n-1}^{(1)} +  \alpha_{21} N_{n-1}^{(1)} = N^{n-1} + \dfrac{\tau}{\mathcal{C}} \gamma_{21} \Pi_{n-1}^{(1)}. $$
Then, if the statement holds for $i=1,2,\dots,\ell-1$, then we prove it for $i=\ell$:
$$ N_{n-1}^{(\ell)} = v_{\ell} N^{n-1} + \dfrac{\tau}{\mathcal{C}} \sum_{j=1}^{\ell-1} \alpha_{\ell j} \Pi_{n-1}^{(j)} + \sum_{j=1}^{\ell-1} \alpha_{\ell j} N_{n-1}^{(j)} = $$
$$ = v_{\ell} N^{n-1} + \dfrac{\tau}{\mathcal{C}} \sum_{j=1}^{\ell-1} \alpha_{\ell j} \Pi_{n-1}^{(j)} + \sum_{j=1}^{\ell-1} \alpha_{\ell j} \left( N^{n-1} + \dfrac{\tau}{\mathcal{C}} \sum_{i=1}^{j-1} \gamma_{ji} \Pi_{n-1}^{(i)} \right) = $$
$$  = \left( v_{\ell} + \sum_{j=1}^{\ell-1} \alpha_{\ell j}\right)  N^{n-1} + \dfrac{\tau}{\mathcal{C}} \sum_{j=1}^{\ell-1} \alpha_{\ell j}\left( \Pi_{n-1}^{(j)} +  \sum_{i=1}^{j-1} \gamma_{ji} \Pi_{n-1}^{(i)} \right) = $$
$$  =  N^{n-1} + \dfrac{\tau}{\mathcal{C}} \sum_{j=1}^{\ell-1} \alpha_{\ell j}\left( \Pi_{n-1}^{(j)} +  \sum_{i=1}^{j-1} \gamma_{ji} \Pi_{n-1}^{(i)} \right) . $$
The last term can be rewritten as
$$ \sum_{j=1}^{\ell-1} \alpha_{\ell j}\left( \Pi_{n-1}^{(j)} +  \sum_{i=1}^{j-1} \gamma_{ji} \Pi_{n-1}^{(i)} \right) = \sum_{j=1}^{\ell-1} \alpha_{\ell j} \Pi_{n-1}^{(j)} + \alpha_{\ell j} \left( \gamma_{j1} \Pi_{n-1}^{(1)}  + \dots + \gamma_{j,j-1} \Pi_{n-1}^{(j-1)} \right) = $$
\begin{align*}
   = \alpha_{\ell 1} \Pi_{n-1}^{(1)} + \alpha_{\ell 2} \Pi_{n-1}^{(2)} + \alpha_{\ell 2} \gamma_{21} \Pi_{n-1}^{(1)} + \alpha_{\ell 3} \Pi_{n-1}^{(3)} + \alpha_{\ell 3} \gamma_{31} \Pi_{n-1}^{(1)} + \alpha_{\ell 3} \gamma_{32} \Pi_{n-1}^{(2)} + \dots + \\
   +\alpha_{\ell,\ell-1} \Pi_{n-1}^{(\ell-1)} + \alpha_{\ell,\ell-1} \left( \gamma_{\ell-1,1} \Pi_{n-1}^{(1)}  + \dots + \gamma_{\ell-1,\ell-2} \Pi_{n-1}^{(\ell-2)} \right)  =
\end{align*}
\begin{align*}
   =\Pi_{n-1}^{(1)} \left( \alpha_{\ell 1} + \alpha_{\ell 2} \gamma_{21} + \alpha_{\ell 3} \gamma_{31} + \dots + \alpha_{\ell,\ell-1 } \gamma_{\ell-1,1} \right) + \\
   + \Pi_{n-1}^{(2)} \left( \alpha_{\ell 2} + \alpha_{\ell 2} \gamma_{22} + \alpha_{\ell 3} \gamma_{32} + \dots + \alpha_{\ell,\ell-1 } \gamma_{\ell-1,2} \right)+ \dots + \alpha_{\ell,\ell-1} \Pi_{n-1}^{(\ell-1)} =
\end{align*}
$$ = \Pi_{n-1}^{(1)} \gamma_{\ell 1} + \Pi_{n-1}^{(2)} \gamma_{\ell 2} + \dots + \Pi_{n-1}^{(\ell-1)} \gamma_{\ell, \ell-1} = \sum_{j=1}^{\ell-1} \gamma_{\ell j} \Pi_{n-1}^{(j)}. $$
Then,
$$ N^{(\ell)} = N^{n-1} + \dfrac{\tau}{\mathcal{C}}  \sum_{j=1}^{\ell-1} \gamma_{\ell j} \Pi_{n-1}^{(j)},$$
    which concludes the proof of the lemma.
\end{proof}

\section*{Acknowledgement}
This research has been supported by the National Research, Development and Innovation Office – NKFIH, grant no. K137699.
The research reported in this paper is also part of project no. BME-NVA-02, implemented with the support provided by the Ministry of Innovation and Technology of Hungary from the National Research, Development and Innovation Fund, financed under the TKP2021 funding scheme.

%
%

\end{document}